\documentclass[a4paper, 11pt]{article}

\usepackage{amssymb, amsthm}
\usepackage[backref]{hyperref}
\usepackage[alphabetic,backrefs,lite]{amsrefs}
\usepackage{amscd}   
\usepackage[all]{xy} 
\usepackage{graphicx}
\usepackage{mathrsfs}
\usepackage{enumerate}

\usepackage{tikz}
\usepackage{verbatim}

  \tikzstyle{block} = [rectangle, draw,
      text width=10em, text centered, minimum height=4em]
  \tikzstyle{line} = [draw, -latex']

\usepackage{amssymb, amsmath, amsthm}
\usepackage[backref]{hyperref}
\usepackage[alphabetic,backrefs,lite]{amsrefs}
\usepackage{amscd}   
\usepackage[all]{xy} 
\usepackage{graphicx}
\usepackage{mathrsfs}
\usepackage{}

\usepackage{tikz}

\newtheorem{lemma}{Lemma}[section]
\newtheorem{theorem}[lemma]{Theorem}
\newtheorem{proposition}[lemma]{Proposition}

\newtheorem*{prop*}{Proposition}
\newtheorem{cor}[lemma]{Corollary}
\newtheorem{conj}[lemma]{Conjecture}

\newtheorem{claim*}{Claim}

\newenvironment{proof3}{\emph{Proof of Theorem 4.1:}}{\hfill$\square$}
\newenvironment{proof4}{\emph{Proof of Theorem 9.4:}}{\hfill$\square$}

\theoremstyle{definition}
\newtheorem{remark}[lemma]{Remark}

\newtheorem*{theoremnon}{Main Theorem}
\newtheorem*{theoremden}{Density Theorem 1}
\newtheorem*{theoremden2}{Density Theorem 2}
\newtheorem*{theoremden3}{Density Theorem 3}

\def\P{\mathbb{P}}

\newcommand{\F}{{\mathbb F}}

\newcommand{\Q}{{\mathbb Q}}

\newcommand{\Z}{{\mathbb Z}}
\newcommand{\N}{{\mathbb N}}

\newcommand{\calA}{{\mathcal A}}
\newcommand{\calB}{{\mathcal B}}
\newcommand{\calC}{{\mathcal C}}
\newcommand{\calD}{{\mathcal D}}

\newcommand{\calG}{{\mathcal G}}
\newcommand{\calH}{{\mathcal H}}

\newcommand{\calO}{{\mathcal O}}

\newcommand{\calU}{{\mathcal U}}

\newcommand{\frakc}{{\mathfrak c}}

\newcommand{\frakf}{{\mathfrak f}}

\newcommand{\frakm}{{\mathfrak m}}

\newcommand{\frakp}{{\mathfrak p}}
\newcommand{\frakq}{{\mathfrak q}}

\newcommand{\frakN}{{\mathfrak N}}

\newcommand{\frakP}{{\mathfrak P}}

\newcommand{\frakS}{{\mathfrak S}}

\DeclareMathOperator{\nr}{Norm}
\DeclareMathOperator{\tr}{Tr}

\DeclareMathOperator{\Frob}{Frob}

\DeclareMathOperator{\Aut}{Aut}
\DeclareMathOperator{\Gal}{Gal}

\DeclareMathOperator{\PGL}{PGL}

\DeclareMathOperator{\SL}{SL}
\DeclareMathOperator{\GL}{GL}

\DeclareMathOperator{\Rad}{Rad}

\usepackage{color}

\numberwithin{equation}{section}
\numberwithin{table}{section}

\title{Asymptotic Generalized Fermat's Last Theorem over Number Fields}

\author{Yasemin Kara, Ekin Ozman}

\newcommand{\Addresses}{{
  \bigskip
  \footnotesize

 Yasemin Kara, \textsc{Bogazici University, Faculty of Arts and Sciences, Mathematics Department, Bebek, Istanbul, 34342, Turkey}\par\nopagebreak
  \textit{E-mail address}, Yasemin Kara: \texttt{yasemin.kara@boun.edu.tr}

   \medskip
   
   Ekin Ozman,  \textsc{Bogazici University, Faculty of Arts and Sciences, Mathematics Department, Bebek, Istanbul, 34342, Turkey}\par\nopagebreak
  \textit{E-mail address}, Ekin Ozman: \texttt{ekin.ozman@boun.edu.tr}
  
  }}

\date{}

\begin{document}

         \maketitle

	\begin{abstract}  Recent work of Freitas and Siksek showed that an asymptotic version of Fermat's Last Theorem holds
	for many totally real fields. Later this result was extended by Deconinck to generalized Fermat equations of the 
	form $Ax^p+By^p+Cz^p=0$, where $A,B,C$ are odd integers belonging to a totally real field.  Another extension was
	given by  \c{S}eng\"{u}n and Siksek. They showed that the Fermat equation holds asymptotically for imaginary quadratic
	number fields satisfying usual conjectures about modularity. 
	
	In this work, combining their techniques we extend their results about the generalized Fermat equations to imaginary 
	quadratic fields. More specifically we prove that the asymptotic generalized Fermat's Last Theorem holds for many
	quadratic imaginary number fields. 
	
		\end{abstract}

\section{Introduction}

 Let $K$ be a number field and $\mathcal O_K$ be its ring of integers. In  \cite{FS}, Freitas and Siksek 
 proved the asymptotic Fermat's Last Theorem over a totally real field $K$.  That is, they
 showed that there is a constant $B_K$ such that for any prime $p>B_K$, the only solutions to the 
 Fermat equation $a^p+b^p+c^p=0$ where $a,b,c \in \calO_{K}$ are the trivial ones satisfying $abc=0$.  They gave 
 an algorithmically testable criterion which, in case satisfied by $K$, implies the asymptotic Fermat's Last Theorem 
 over $K$ and they proved that their criterion is satisfied by real quadratic fields
 $K=\Q(\sqrt{d})$ for a subset of $d\geq2$ with density $5/6$ among the other squarefree positive integers.

 Then, Deconinck \cite{HD} extended the results of Freitas and Siksek \cite{FS} to the generalized Fermat equation of the
 form $Aa^p+Bb^p+Cc^p=0$ where $A,B,C$ are odd integers belonging to a totally real field.  In a recent paper \cite{SS},
 \c{S}eng\"{u}n and Siksek proved the asymptotic Fermat's Last Theorem for any number field $K$ by assuming
 two deep but standard conjectures form the Langlands programme.  They also deduced that the asymptotic Fermat's
 Last Theorem holds for imaginary quadratic fields $\Q{\sqrt {-d}}$ with $-d\equiv 2,3 \pmod 4$ squarefree.\\
 
 The purpose of this paper is to extend the results of \cite{SS} to the generalized Fermat equation of the form
 $Aa^p+Bb^p+Cc^p=0$ where $A,B,C$ are odd integers belonging to any number field $K$.  In addition, we prove three density
 results, similar to the ones given in \cite{SS}, \cite{HD} and \cite{FS}.
 
 \subsection*{Our results}
 
 Before we state our results, we set up some notations.
 Let $K$ be a number field and $\calO_{K}$ be its ring of intergers.  Let $A, B, C$ be non-zero elements of $\calO_{K}$,
 and $p$ be a prime.  We refer the equation
 \begin{equation}\label{maineqn}
  Aa^p+Bb^p+Cc^p=0,\quad a,b,c \in \calO_{K}
 \end{equation}
as \textit{the generalized Fermat equation over K with coefficients A, B, C and exponent p}.  A solution $(a,b,c)$ is
called \textbf{trivial} if $abc=0$, otherwise \textbf{non-trivial}.  We set up the following notation throughout the paper.

\begin{itemize}
\item $
 R:=\Rad(ABC)=\prod_{\begin{subarray}{c}
                    \frakq\;|\;  ABC \\
                    \frakq\;\footnotesize\text{{prime in}}\; \mathcal O_K
                    \end{subarray}}\frakq,  $
                    
                    \item $S:=\{\frakP : \frakP\;\mbox{is a prime ideal of}\; \calO_{K}\;\mbox{such that}\; \frakP|2R\},$  \label{sets}
 \item $T:=\{\frakP : \frakP\;\mbox{is a prime ideal of}\; \calO_{K}\;\mbox{above}\; 2\},$ 
\item $ U:=\{\frakP\in T : f(\frakP/2)=1\}.$
\end{itemize}

 Our main theorem depends on two deep but standard conjectures, which will be explained in the following section. 
 These conjectures are the same ones assumed in \cite{SS}. Their assumption is necessary since the analogues of Modularity 
 Theorem have not proved yet in general. 

\begin{theoremnon}
 Let $K$ be a number field satisfying Conjectures \ref{conj1} and \ref{conj2}.  Let $A,B,C \in \mathcal{O}_K$ and suppose that
 $A,B,C$ are odd, in the sense that if $\frakP$ is a prime of $\calO_{K}$ lying over $2$, then $\frakP\;\nmid\;ABC$.  Write
 $\calO_{S}^{*}$ for the set of $S$-units of $K$.  Suppose that for every solution $(\lambda,\mu)$ to the $S$-unit
 equation
 \begin{equation}\label{uniteqn}
 \lambda+\mu=1,\quad\quad \lambda,\mu\in\calO_{S}^{*},
 \end{equation}
 there is some $\frakP\in U$ that satisfies $\max\{|v_{\frakP}(\lambda)|,|v_{\frakP}(\mu)|\}\leq 4v_{\frakP}(2)$.  Then
 there is a constant $\calB=\calB(K,A,B,C)$ such that the generalized Fermat equation with exponent $p$ and coefficients
 $A,B,C$ (given in equation \ref{maineqn}) does not have non-trivial solutions with $p>\calB$. We refer to this 
 as ``\emph{the asymptotic generalized Fermat's Last Theorem holds for $K$}.'' 
\end{theoremnon}

Also, we prove three density results.

\begin{theoremden}
Let $K=\mathbb Q(\sqrt{-d})$ be an imaginary quadratic field where $d$ is a squarefree positive integer such
that $-d \equiv 2,3 \pmod 4$. Let $q \geq 29$ be a prime such that $q \equiv 5 \pmod 8$ and 
$\left ( \frac{-d}{q} \right ) = -1$. Assume Conjectures \ref{conj1} and \ref{conj2}. Then there exists a constant 
depending on $K$ and $q$, namely $B_{K,q}$, such that for all $p > B_{K,q}$ the Fermat equation $x^p+y^p+q^rz^p=0$ doesn't 
have any non-trivial solutions. 
\end{theoremden}

\begin{theoremden2} Let $K=\mathbb Q(\sqrt{-d})$ be an imaginary quadratic field where $d$ is a squarefree positive integer
such that $d \equiv 7 \pmod 8, d \equiv 5 \pmod 6$ and $ d \not\equiv 7 \pmod {14}$.  Assume Conjectures \ref{conj1} and \ref{conj2}. 
Then there exists a constant 
depending on $K$, namely $B_{K}$, such that for all $p > B_{K}$ the Fermat equation $x^p+y^p+z^p=0$ doesn't 
have any nontrivial solutions.  We refer to this 
 as ``\emph{the asymptotic Fermat's Last Theorem holds for $K$}.''  

\end{theoremden2}


\begin{theoremden3}
Assuming Conjectures \ref{conj1} and \ref{conj2} the asymptotic Fermat's Last Theorem holds for $5/6$ of the imaginary 
quadratic number fields. 
\end{theoremden3}

 As mentioned in the Introduction our results will be a generalization of the main theorem in \cite{SS}. Therefore we follow 
 their expostion and proofs very closely. We try to avoid repetition however when necessary we include the statements and/or
 proofs of the results proved in \cite{SS}.

\subsection*{Acknowledgements}

We thank Professor Samir Siksek for useful discussions and helpful comments.\\
Both authors are supported by T\"{U}B\.{I}TAK (Turkish National and Scientific Research Council) Research Grant 117F045.
\section{Preliminaries}

In this section we discuss modular forms and set up the terminology and notation which are necessary to state the conjectures
assumed in the main theorem. For more details we refer to Sections 2 and 3 of \cite{SS} and the references therein.

\subsection{Background about Eigenforms and Galois Representations}

Let $K$ be a number field with ring of integers $\calO_{K}$ and $\frakN$ be an ideal of $\calO_{K}$.  A \emph{weight two 
complex eigenform $f$ over $K$ of degree $i$ and level $\frakN$} is a ring homomorphism 
$f: \mathbb T_{\mathbb C}^{(i)} (\frakN) \rightarrow \mathbb C$ where $T_{\mathbb C}^{(i)} (\frakN) $ is the commutative 
$\mathbb Z$-algebra inside the endomorphism algebra of $H^i(Y_0(\frakN), \mathbb C )$ generated by the Hecke operators 
$T_\frakq$ for every prime $\frakq$ not dividing $\frakN$. For the precise definition of the locally symmetric space
$Y_0(\frakN)$ we refer to \cite{SS}. Similarly for primes $p$ that are unramified in $K$ and relatively prime to $\frakN$ one 
can define the \emph{weight two mod $p$ eigenform $\theta$ over $K$ of degree $i$ and level $\frakN$} as a ring homomorphism 
$\theta: \mathbb{T}_{\bar{\mathbb{F}_p}}^{(i)}(\frakN)  \rightarrow \bar{\mathbb{F}}_p.$ In this case only the Hecke operators
$T_\frakq$ are considered for the primes $\frakq$ relatively prime to $p\frakN$.

Let $\theta$ be a mod $p$ eigenform of level $\frakN$ and degree $i$. Let $\mathbb Q_\frakf$ be the number field generated by
the values of $\frakf$. If there is a complex eigenform $\frakf$ of degree $i$ and level $\frakN$ and a prime ideal $\frakp$ of
$\mathbb Q_\frakf$ lying over $p$ such that $\theta(T_\frakq) \equiv \bar{\frakf(T_\frakq)} \mod \frakp$ for every prime
$\frakq$ of $K$ coprime to $p\frakN$, we say that \emph{$\theta$ lifts to a complex eigenform.}

The following result is proved by \c{S}eng\"{u}n and Siksek in \cite{SS}.

\begin{prop*}(Proposition 2.1, \cite{SS})
 There is an integer $B$, depending only on $\frakN$, such that for any prime $p>B$, every $\mod p$ eigenform of level
 $\frakN$ lifts to a complex one.
\end{prop*}

\subsection{Conjectures} One of the conjectures assumed in the main theorem is a special case of Serre's modularity conjecture 
over number fields, stated below:

\begin{conj}(Conjecture 3.1 in \cite{SS})\label{conj1}  Let $\rho:G_K\rightarrow GL_2(\bar{\mathbb{F}}_p)$ be an odd, 
irreducible, continuous representation with 
Serre conductor $\frakN$ 
and trivial character. 
Assume that $p$ is unramified in $K$ and that $\bar{\rho}|_{G_{K_{\frakp}}}$ arises from
a finite-flat group scheme over $\calO_{K_{\frakp}}$ for every prime $\frakp|p$.  Then there is a weight 2 $\mod p$
eigenform $\theta$ over $K$ of level $\frakN$ such that for all primes $\frakq$ coprime to $p\frakN$, we have
\[
 \tr(\bar{\rho}(\Frob_{\frakq}))=\theta(T_{\frakq}).
\]
\end{conj}

 Additionally, we will make use of a special case of a fundamental conjecture from Langlands Programme.
 
 \begin{conj}\label{conj2}(Conjecture 4.1 in \cite{SS})
  Let $\frakf$ be a weight 2 complex eigenform over $K$ of level $\frakN$ that is non-trivial and new.  If $K$ has some
  real place, then there exists an elliptic curve $E_{\frakf}/K$ of conductor $\frakN$ such that 
  \begin{equation}\label{c2eqn}
   \#E_{\frakf}(\calO_K/\frakq)=1+\nr(\frakq)-\frakf(T_{\frakq})\quad\mbox{for all}\quad\frakq\;\nmid\;\frakN.
  \end{equation}
 If $K$ is totally complex, then there exists either an elliptic curve $E_{\frakf}$ of conductor $\frakN$ satisfying (\ref{c2eqn})
 or a fake elliptic curve 
 $A_{\frakf}/K$, of conductor $\frakN^2$, such that
 \begin{equation}
  \#A_{\frakf}(\calO_K/\frakq)=(1+\nr(\frakq)-\frakf(T_{\frakq}))^2\quad\mbox{for all}\quad\frakq\;\nmid\;\frakN.
 \end{equation}

 \end{conj}

\section{Frey Curve and Related Facts}

In this section we collect some facts related to the Frey curve associated to a proposed solution of the Fermat equation 
\ref{maineqn} and the associated Galois representation.

Let $G_{K} $ be the absolute Galois group of $K$ and $E$ be an elliptic curve over a number field $K$. Then
\[
 \overline{\rho}_{E,p}:G_{K}\longrightarrow \Aut(E[p])\cong\GL_{2}(\mathbb{F}_{p})
\]
denotes the mod $p$ Galois representation of $E$.  

The below lemma can be found in \cite{SS}.  We include its statement here for the convenience of the reader.
\begin{lemma}\label{sizeofinertia}
 Let $E$ be an elliptic curve over $K$ with $j$-invariant $j$.  Let $p\geq 5$ be a rational prime.  Let $\frakq \nmid p$ be 
 a prime of $K$.
 \begin{enumerate}[(i)]
 \item If $v_{\frakq}(j)\geq 0 $ (i.e. $E$ has potentially good reduction at $\frakq$) then 
 the size of $\overline{\rho}_{E,p}(I_{\frakq})$ divides $24.$
 \item Suppose $v_{\frakq}(j)<0$ (i.e. $E$ has potentially multiplicative reduction at $\frakq$).
  \begin{itemize}
   \item If $p\;\nmid \;v_{\frakq}(j) $ then the size of $\overline{\rho}_{E,p}(I_{\frakq})$ is either $p$ or $2p$.
   \item If $p\;| \;v_{\frakq}(j) $ then the size of $\overline{\rho}_{E,p}(I_{\frakq})$ is $1$ or $2$.
  \end{itemize}
  \end{enumerate}
\end{lemma}

Let $T$ and $U$ be the sets defined in the introduction i.e. let $\mathcal S_K$ be the set of prime ideals of $\mathcal O_K$ 
and $R=\Rad{(ABC)}$. Then $S=\{\frakP \in \mathcal S_K: \; \frakP|2R\}, T=\{\frakP \in \mathcal S_K: \; \frakP | 2 \}$ and 
$U=\{\frakP\in T : f(\frakP/2)=1\}$. Further, suppose that $U\neq \emptyset$.  Let $p$ be an odd prime
and $(a,b,c)\in K^3$ be a non-trivial solution to the Fermat equation \ref{maineqn}.  Write 
\[
 \calG_{a,b,c}=a\calO_K+b\calO_K+c\calO_K
\]
which we can think of as the greatest common divisor of $a, b, c$.  We denote the class of $\calG_{a,b,c}$ in the class group 
of $K$,  by
$[a,b,c]$.  Let $\mathbf{\frakc}_{1},\dots,\mathbf{\frakc}_{h}$ denote the ideal classes of $K$.  Since every ideal class
contains infinitely many prime ideals, for each class $\mathbf{\frakc}_{i}$ we can fix a prime ideal
$\mathbf{\frakm}_{i}$ coprime to $2R$ with the smallest possible norm reperesenting $\mathbf{\frakc}_{i}$.  
The set $\calH=\{\mathbf{\frakm}_{i},\dots,\mathbf{\frakm}_{h}\}$ denotes our fixed choice of odd prime ideals
representing the class group.  Note that it is possible to scale $(a,b,c)$ so that it remains integral and 
$\calG_{a,b,c}=\frakm\in \calH$.
For a solution $(a,b,c)$ of the Fermat equation \ref{maineqn}, we associate the Frey elliptic curve
\begin{equation} \label{Frey}
 E=E_{a,b,c}: Y^2=X(X-Aa^p)(X+Bb^p).
\end{equation}


\begin{lemma}\label{Udiv}
 Let $\frakP\in U$ and suppose $p>4v_{\frakP}(2)$.  Then
 \begin{enumerate}[(i)]
 \item $E$ has potentially multiplicative reduction at $\frakP$;
 \item $p$ divides the size of $\overline{\rho}_{E,p}(I_{\frakP})$ where $I_{\frakP}$ denotes the inertia subgroup $G_{K}$ 
 at $\frakP$.
 \end{enumerate}
\end{lemma}

\begin{proof}
 Since $\calG_{a,b,c}=\frakm$ is relatively prime to $2R$ and is in $\mathcal H$, which consists of odd prime ideals, 
 we have $\frakm\;\nmid\;2$ and $\frakm\;\nmid\;R$.  Furthermore, $\frakP\;\nmid\;R$
 implies that $\frakP$ divides at most one of $a,b,c$.  Say $\frakP\;\nmid\;abc$, 
 then we get, $0=Aa^p+Bb^p+Cc^p\equiv 1+1+1 \mod \frakP $. Since the residue field of any element of $U$ is $\mathbb{F}_2$, 
 this is a contradiction.  Therefore, we see that $\frakP$ divides precisely one of $a,b,c$.  Now, we permute $a,b,c$ so that 
 $\frakP\;|\;b$.  The Frey curve associated to $(b,a,c)$ is
 \[
  E^{\prime}=E_{b,a,c}: Y^2=X(X-Ab^p)(X+Ba^p).
 \]
For $E$ and $E^{\prime}$ we will compute the corresponding $j$ and $j^{\prime}$.  The expressions for $j$ and $j^{\prime}$
in terms of $A,B,C$ and $a,b,c$ are
\[
 j=j^{\prime}=2^8.\frac{(C^2c^{2p}-ABa^p b^p)^3}{A^2 B^2 C^2 a^{2p}b^{2p}c^{2p}}
\]
 Since $\frakP\;\nmid 2R$, we know that $\frakP\; \nmid ABC$.  Therefore, it follows that
 $v_{\frakP}(j)=8v_{\frakP}(2)-2pv_{\frakP}(b)$ and $v_{\frakP}(j)=8v_{\frakP}(2)-2pv_{\frakP}(a)$.  By assumption 
 $p>4v_{\frakP}(2)$ which implies that $v_{\frakP}(j)<0$.  Hence, $E$ has multiplicative reduction at $\frakP$.
 Furthermore, $p\; \nmid v_{\frakP}(j)$.  The rest of the lemma follows from the part $(ii)$ of Lemma \ref{sizeofinertia}.
 \end{proof}
 
 \begin{lemma}\label{lem:24}
  Suppose $p\geq 5$ and $\frakm\;\nmid p$.  Then the size of $\overline{\rho}_{E,p}(I_{\frakm})$ divides $24$.
 \end{lemma}

\begin{proof}
 We know that $v_{\frakm}(a), v_{\frakm}(b)$ and $v_{\frakm}(c)$ are all positive as $\calG_{a,b,c}=\frakm$.
 Furthermore, since $Aa^p+Bb^p+Cc^p=0$, we have $v_{\frakm}(-Cc^p)=v_{\frakm}(Aa^p+Bb^p)
 \geq \min\{v_{\frakm}(Aa^p),v_{\frakm}(Bb^p)\}$ with equality if $v_{\frakm}(Aa^p)\neq v_{\frakm}(Bb^p)$,
 which implies that at least two of $v_{\frakm}(Aa^p), v_{\frakm}(Bb^p),
 v_{\frakm}(Cc^p)$ are equal.  Then, either we have 
 \[
  v_{\frakm}(Aa^p)=v_{\frakm}(-Cc^p)=k,\quad v_{\frakm}(Bb^p)=k+t
 \]
or
 \[
  v_{\frakm}(Bb^p)=v_{\frakm}(-Cc^p)=k,\quad v_{\frakm}(Aa^p)=k+t
 \]
 where $k\geq 1$ and $t\geq 0$.  From the expression of $j$, we get 
 \[\begin{matrix}
  
  v_{\frakm}(j)&\geq 3\min\{2v_{\frakm}(Cc^p), v_{\frakm}(Aa^p)+v_{\frakm}(Bb^p)\}\\
                &-2(v_{\frakm}(Aa^p)+v_{\frakm}(Bb^p)
  +v_{\frakm}(Cc^p)). \end{matrix}
 \]
If $t=0$, then $v_{\frakm}(j)=0$ and thus the lemma follows from the part $(i)$ of Lemma \ref{sizeofinertia}.  Now, suppose 
that $t\geq 1$.
In this case, we get $v_{\frakm}(j)=-2t$.  If we can show that $p\;|-2t$, by Lemma \ref{sizeofinertia}, the size of
$\overline{\rho}_{E,p}(I_{\frakm})$ is $1$ or $2$, completing the proof.

Since $k=v_{\frakm}(Aa^p)=v_{\frakm}(A)+pv_{\frakm}(a)$ and $k+t=v_{\frakm}(Bb^p)=v_{\frakm}(B)+pv_{\frakm}(b)$, we have
$t=v_{\frakm}(B)-v_{\frakm}(A)+p(v_{\frakm}(b)-v_{\frakm}(a))$.  We also know that $\frakm\;\nmid R$, so 
$\frakm\;\nmid ABC$.  Hence, $\frakm$ divides none of the constants $A,B,C$ which gives
$v_{\frakm}(A)=v_{\frakm}(B)=v_{\frakm}(C)=0$.

\end{proof}

Let $E$ be an elliptic curve over a number field having semistable reduction away from a set of primes 
$\mathcal P$. If $p \geq 5$ is not divisible by any prime in $\mathcal P$ then by the theory of Weil pairing on $E[p]$ the 
determinant of the mod $p$ Galois representation associated to $E$, namely $\bar{\rho}_{E,p}$  has determinant equal to the 
cyclotomic character, [see \cite{GJG}, p.21]. Therefore  the representation $\bar{\rho}_{E,p}$ is odd. Recall that a 
representation $\bar{\rho}: G_K \rightarrow \GL_2(\bar{\mathbb F}_p)$ is \emph{odd} if the determinant of every complex 
conjugation is $-1$. Given a number field $K$, we obtain a \emph{complex conjugation} for every real embedding 
$\sigma: K \hookrightarrow \mathbb R$ and every extension $\tilde{\sigma}: \bar{K} \hookrightarrow \mathbb C$ of 
$\sigma$ as $\tilde{\sigma}^{-1}\iota \tilde{\sigma} \in G_K$ where $\iota$ is the usual complex conjugation. If the number 
field $K$ has no real embeddings to begin with, then we immediately say that $\bar{\rho}$ is odd.

The following results gives some of the reduction properties of the Frey elliptic curve. 

\begin{lemma}\label{Freyred}
The Frey curve $E$ associated to the Fermat equation \ref{maineqn} is semistable away from $S\cup\{\frakm\}$, where 
$\frakm=\calG_{a,b,c}$. 
 \end{lemma}
 
 \begin{proof}
Let $\frakq\notin S\cup\{\frakm\}$ be a prime of $K$.
  The invariants $c_4$ and $\Delta$ of the model $E$ given in (\ref{Frey}) are
  \[
   c_4=2^4((Bb^p)^2-Aa^{p}Cc^{p}),\quad \Delta=2^4(ABC)^2(abc)^{2p}.
  \]
 Suppose that $\frakq$ divides $\Delta$.  None of $2,A,B$ and $C$ is divisible by $\frakq$ since
 $\frakq\;\nmid\;2R$ which implies that $abc$ is divisible by $\frakq$.  On the other hand, $\frakq$ can divide
 at most one of $a,b,c$.  If two of them were divisible by $\frakq$, then $Aa^p+Bb^p+Cc^p=0$ would lead to $\frakq=\frakm$
 which is a contradiction.  Therefore, $((Bb^p)^2-Aa^{p}Cc^{p})$ is not divisible by $\frakq$ and so $\frakq\;\nmid\;c_4$.
 This also gives $v_{\frakq}(c_4)=0$.  Hence, the given model is minimal and $E$ is semistable at $\frakq$.
 \end{proof}
 
 \begin{cor}\label{cor:finite}
 Let $E$ be the Frey curve associated to the Fermat equation \ref{maineqn}, defined over the number field $K$. Let $p \geq 5$ 
 and $p$ is not divisible by any $\frakq \in S \cup \{\frakm\}$ and  let $\bar{\rho}_{E,p}$ denote the mod $p$ Galois 
 representation associated to $E$. Then:
 \begin{enumerate}
 \item \label{odd}  The determinant of $\bar{\rho}_{E,p}$  is the mod $p$ cyclotomic character, hence it is odd.
 \item \label{Serrecon} Serre conductor of $\bar{\rho}_{E,p}$, $\frakN$, is supported on
 $S\cup\{\frakm\}$ and belongs to a finite set that depends only on the
  field $K$. 
  \item \label{finflat} $\bar{\rho}_{E,p}$ is finite flat at every $\frakq$ over $p$.
 \end{enumerate} 
 \end{cor}

\begin{proof}
 Part (\ref{odd})  follows from the paragraph preceding Lemma \ref{Freyred}.   

For Parts (\ref{Serrecon}) and (\ref{finflat})we use $\Delta=2^4(ABC)^2(abc)^{2p}$ computed in Lemma \ref{Freyred}. 
Now, $v_{\frakq}(\Delta)=4v_{\frakq}(2)+2(v_{\frakq}(A)+v_{\frakq}(B)+v_{\frakq}(C))+
                         2p(v_{\frakq}(a)+v_{\frakq}(b)+v_{\frakq}(c))=2pv$.
where $v$ is exactly one of the $v_{\frakq}(a),v_{\frakq}(b),v_{\frakq}(c)$.  Hence, $p\;|\;v_{\frakq}(\Delta)$.  
We deduce that $\bar{\rho}$ is unramified at $\frakq$ if $\frakq\;\nmid\;p$ and finite flat at $\frakq$ 
if $\frakq\;|\;p$ (c.f.\cite{Serre}).  Finally, we show that there can be only finitely many Serre conductors $\frakN$. 
Only the primes
in $\frakq\in S\cup\{\frakm\}$ can divide $\frakN$.  As $\frakN$ divides the conductor $N$ of $E$, 
$v_{\frakq}(\frakN)\leq v_{\frakq}(N)\leq 2+3 v_{\frakq}(3)+6 v_{\frakq}(2)$ by [Silverman,theorem IV.10.4].
Hence, there can be only finitely many Serre conductors and they only depend on $K$ since $\frakm\in\calH$ and 
$\calH$ only depends on $K$.
\end{proof}

\section{Level Lowering}

In this section we will be relating the Galois representation attached to Frey curve with another  representation of lower 
level. The following result is the main theorem of the section:

\begin{theorem}\label{thm:levelred}
 Let $K$ be a number field.  Assume Conjecture \ref{conj1} and Conjecture \ref{conj2}.  Suppose that $U\neq\emptyset$.  
 Then there is a 
 constant $B_K$ depending only on $K$ such that the following holds.  Let $(a,b,c)\in \calO_K^3$ be a non-trivial
 solution to the Fermat equation with exponent $p>B_K$, and suppose that it is scaled so that 
 $\calG_{a,b,c}=\frakm\in\calH$.  Let $E/K$ be the associated Frey curve defined in \eqref{Frey}.  Then there is an
 elliptic curve $E'/K$ such that the following statements hold:
 
 \begin{enumerate}[(i)]
 \item $E'$ has good reduction away from $S\cup\{\frakm\}$, and potentially good reduction away from $S$.
 \item $E'$ has full 2-torsion.
 \item $\bar{\rho}_{E,p}\sim\bar{\rho}_{E',P}$.
 \item For $\frakP\in U$ we have $v_{\frakP}(j')<0$ where $j'$ is the $j$-invariant of $E'$.
 \end{enumerate}
\end{theorem}
We will give the proof of this Theorem in Section \ref{appcon} after stating the necessary lemmas. 
\subsection{Surjectivity of Galois representations}

The following is Proposition 6.1 from \c{S}eng\"{u}n and Siksek \cite{SS}. We include its statement for the convenience of the  
reader but we will omit its proof and refer to \cite{SS} instead.
\begin{proposition}
Let $L$ be a Galois number field and let $\frakq$ be a prime of $L$. There is a constant $B_{L,\frakq}$ such that 
the following is true. Let $p > B_{L, \frakq}$ be  a rational prime. Let $E/L$ be an elliptic curve that is semistable 
at all $\frakp | p$ and having potentially multiplicative reduction at $\frakq$. Then $\bar{\rho}_{E,p}$ is irreducible. 

\end{proposition}

By applying the above proposition to the Frey curve, we get the following corollary.

\begin{cor}\label{galrepsur}
 Let $K$ be a number field, and suppose that $U\neq\emptyset$.  There is a constant $C_K$ such that if $p>C_K$ and 
 $(a,b,c)\in \calO_{K}^3$ is a non-trivial solution to the Fermat equation with exponent $p$, and scaled so that
 $\calG_{a,b,c}\in \calH$, then  $\bar{\rho}_{E,p}$ is surjective, where $E$ is the Frey curve given in (\ref{Frey}).
 
\end{cor}

\begin{proof}
 Suppose $\frakP\in U$, then $E$ has potentially multiplicative reduction at $\frakP$ by Lemma 4.2.  Also, $E$ is semistable
 away from the primes in $S\cup\calH$ from Lemma 4.4.  Let $L$ be the Galois closure of $K$, and let $\frakq$ be a prime
 of $L$ above $\frakP$.  Now, by applying Proposition 5.1, we get a constant $B_{L,\frakq}$ such that 
 $\bar{\rho}_{E,p}(G_L)$ is irreducible whenever $p>B_{L,\frakq}$.  Now, we observe that $B_{L,\frakq}$ depends on only
 $K$ since $\frakp$ is a prime of $L$ above $\frakP$ which is a prime of $K$ above 2.  We denote $B_{L,\frakq}$ by $C_K$.
 If necessary, enlarge $C_{K}$ so that $C_{K}>4v_{\frakP}(2)$.  Now, we apply part (ii) of Lemma 4.2 and see that
 the image of $\bar{\rho}_{E,p}$ contains an element of order $p$.  It is known that any subgroup of $\GL_2(\F_p)$ 
 having an element of order $p$ is either reducible or contains $\SL_2(\F_p)$.  As $p>C_K>4v_{\frakP}(2)$, the image
 contains $\SL_2(\F_p)$.  Finally, we can ssume that $K\cap\Q(\zeta_p)=\Q$ by taking $C_K$ large enough if needed.
 Hence, $\chi_p=\det(\bar{\rho})_{E,p}$ is surjective.
\end{proof}

\label{appcon} \subsection{Proof of Theorem \ref{thm:levelred}} 

In this section, Theorem \ref{thm:levelred} will be proven. We continue with the notations introduced in the statement of
Theorem \ref{thm:levelred} and the assumptions of the theorem.

 \begin{lemma}\label{reduced}
 There is a non-trivial, new (weight 2) complex eigenform $\frakf$ which has an associated elliptic curve 
  $E_{\frakf}/K$ of conductor $\frakN'$ dividing $\frakN$.
 \end{lemma}


 \begin{proof}
 
 To start with, we show the existence of such an eigenform $\frakf$ of level $\frakN$ supported only on 
  $S\cup\{\frakm\}$.
  
   By Corollary \ref{galrepsur}, the representation $\bar{\rho}_{E,p}:G_K\rightarrow GL_{2}(\mathbb{F}_p) $ is surjective 
   hence is absolutely irreducible for $p>C_K$.  Now, we apply  Conjecture \ref{conj1} to deduce that there is 
   a weight 2 $\mod p$ eigenform $\theta$ 
   over $K$ of level $\frakN$ such that for all primes $\frakq$ coprime to $p\frakN$, we have 
   \[\tr(\bar{\rho}_{E,p}(\Frob_{\frakq}))=\theta(T_{\frakq}).\]
 We know from Corollary \ref{cor:finite} that there are only finitely many possible levels $\frakN$.  Thus, by taking $p$
 large enough,
 there is a weight 2 complex eigenform $\frakf$ with level $\frakN$ which is a lift of $\theta$ for any level $\frakN$.
 Note that since there are only finitely many such eigenforms $\frakf$
  and they depend only on $K$, from now on we can suppose that every constant depending on these eigenforms depends
 only on $K$.
  
   Next, we recall that if $\Q_{\frakf}=\Q$ then  there is a constant $C_{\frakf}$ depending only on $\frakf$ such that 
   $p<C_{\frakf}$ (\cite{SS}, Lemma 7.2).  Therefore, by taking $p$ sufficiently large, we assume that $\Q_{\frakf}=\Q$. 
   In order to apply Conjecture \ref{conj2}, we need to
  show that $\frakf$ is non-trivial and new.  As $\bar{\rho}_{E,p}$ is irreducible, the eigenform $\frakf$ is non-trivial.  
  If $\frakf$ is new, we are done.  If not, we can replace it with an equivalent new eigenform of smaller level.  Therefore,
  we can take $\frakf$ new with level $\frakN'$ dividing $\frakN$.
  Finally, we apply Conjecture \ref{conj2} and obtain that $\frakf$ either has an associated elliptic curve $E_{\frakf}/K$ 
  of conductor $\frakN'$, or has an associated fake elliptic curve $A_{\frakf}/K$ of conductor $\frakN^2$.
  
  By Lemma \ref{fake},  if $p>24$, then $\frakf$ has an associated elliptic curve $E_{\frakf}$.

    As a result, we can assume that $\bar{\rho}_{E,p}\sim \bar{\rho}_{E',p}$ where $E'=E_{\frakf}$ is an elliptic
    curve with conductor $\frakN'$ dividing $\frakN$.  This completes the proof of the claim.

      \end{proof}

      \begin{lemma}\label{fake}
       If $p>24$, then $\frakf$ has an associated elliptic curve $E_{\frakf}$.
      \end{lemma}
      \begin{proof}  
      
      Let $\frakP\in U$ as we assume $U\neq\emptyset$.  By Lemma 4.2, $p\;|\;\#\bar{\rho}_{E,p}(I_{\frakP})$
       whenever $p>4v_{\frakP}(2)$.  Assume for a contradiction that $\frakf$ corresponds to a fake elliptic curve $A_{\frakf}$, then
       $\#\bar{\rho}_{A_{\frakf},p}(I_{\frakP})\leq 24$ by Theorem 4.2 of \cite{SS}.  (Note that here, $\bar{\rho}_{A_{\frakf},p}$
       is the 2-dimensional representation defined in Section 4, Equation 3 of \cite{SS}).  Thus 
       $\#\bar{\rho}_{E,p}(I_{\frakP})\leq 24$ due to
       $\bar{\rho}_{A_{\frakf},p}\sim \bar{\rho}_{E,p}$ leading to a contradiction when $p>24$.
      \end{proof}

      We can now give the proof of Theorem \ref{thm:levelred}.

 \begin{proof3}
Lemma \ref{fake} gives us that if $p>24$, then $\frakf$ has an associated elliptic curve $E_{\frakf}$.
  Therefore, by Lemma \ref{reduced} we can assume that $\bar{\rho}_{E,p}\sim\bar{\rho}_{E',p}$ where $E'=E_{\frakf}$ is an
  elliptic curve
  of conductor $\frakN'$ dividing $\frakN$. 

  The elliptic curve $E'$ does not need to have full 2-torsion.  We will refer to \cite{SS} to overcome this problem.
  If $E'$ does not have full 2-torsion, we can always find another elliptic curve $E''$ with full 2-torsion that is
  2-isogenous to $E'$ for big enough prime $p$, see Lemma 7.4, Lemma 7.5 in \cite{SS}.

 Now, suppose $E'$ is $2$-isogenous to an elliptic curve $E''$.  As the isogeny induces an isomorphism
 $E'[p]\cong E''[p]$ of Galois modules ($p\neq 2$), $\bar{\rho}_{E,p}\sim \bar{\rho}_{E',p}\sim\bar{\rho}_{E'',p}$ 
 completing the proof of (iii).  After possibly replacing $E'$ by $E''$, we can suppose that $E'$ has full 2-torsion 
 giving us (ii).
 
 It remains to prove (i) and (iv).  Let $\frakP\in U$, then $p$ divides the size of $\bar{\rho}_{E,p}(I_{\frakP})$ by Lemma
 \ref{Udiv}.  Now,
 Lemma \ref{sizeofinertia} implies that $v_{\frakP}(j')<0$ for $\frakP\in U$ since the sizes of  
 $\bar{\rho}_{E,p}(I_{\frakP})$ and $\bar{\rho}_{E',p}(I_{\frakP})$ are equal.  This
 proves (iv).  Finally, we need to show that $E'$ has potentially good reduction at $\frakm$.  Assume the contrary that
 $E'$ has a multiplicative reduction at $\frakm$.  Then, by Lemma \ref{sizeofinertia}, we have
 $p\;|\;\#\bar{\rho}_{E',p}(I_{\frakm})$ for
 every $p>|v_{\frakm}(j')|$.  On the other hand, the size of $\bar{\rho}_{E',p}(I_{\frakm})$ is less than $24$ by 
 Lemma \ref{lem:24} giving a contradiction
 for large $p$.
 
  \end{proof3}

\section{Relation with S-unit equation and elliptic curves}
So far, we have proved that assuming a non-trivial solution to the generalized Fermat equation yields an elliptic curve
$E'$ with full 2-torsion having potentially good reduction away from the set $S$.  In order to prove our Main Theorem,
we relate such elliptic curves $E'$ to the solutions of the $S$-unit equation \ref{uniteqn} via $\lambda$-invariants of 
elliptic curves (see, e.g., \cite{Sil86}, pp. 53-55).  For more details, we refer the reader to the Section 5 of \cite{HD}.

Let $E'$ bea an elliptic curve over $K$ with full 2-torsion which has a model
\begin{equation}\label{2torcurve}
E` : Y^2=(X-e_1)(X-e_2)(X-e_3),
\end{equation} 
where $e_1, e_2, e_3$ are all distinct.  Then their cross ratio
$ \lambda=\frac{e_3-e_1}{e_2-e_1} $ is an element of $\P^1(K)-\{0,1,\infty\}.$ Conversely we can obtain any $\lambda$ in $\P^1(K)-\{0,1,\infty\}$ as a
cross ratio of three distinct elements of $K$.  Now, let us denote the symmetric group on three letters by $\frakS_3$.
The action of $\frakS_3$ on the set $\{e_1,e_2,e_3\}$ can be extented to a well-defined action on $\P^1(K)-\{0,1,\infty\}$
via the cross ratio.  Hence, we can view $\frakS_3$ as the below subgroup of $\PGL_2(K)$:
\[
 \frakS_3=\{z,1/z,1-z,1/(1-z),z/z-1,(z-1)/z\}.
\]

Under the action of $\frakS_3$, the orbit of $\lambda$ in $\P^1(K)-\{0,1,\infty\}$, called \textit{$\lambda$-invariants}
corresponds to the following set:
\begin{equation}\label{lambdainvariants}
 \left\{\lambda,\frac{1}{\lambda},1-\lambda,\frac{1}{1-\lambda},\frac{\lambda}{\lambda-1},\frac{\lambda-1}{\lambda}\right\}.
\end{equation}
Every elliptic curve $E'$ of the form (\ref{2torcurve}) is isomorphic (over $\bar{K}$) to an elliptic curve in the 
\textit{Legendre form} $E_{\lambda} : Y^2=X(X-1)(X-\lambda)$ for $\lambda\in\P^1(K)-\{0,1,\infty\}$ whose $j$-invariant 
is 
\begin{equation}\label{j-inv}
j(\lambda)=2^8\frac{(\lambda^2-\lambda+1)^3}{\lambda^2(1-\lambda)^2}.
\end{equation}
Moreover, there is a one-to-one correspondence between the $\bar{K}$-isomorphism classes of the Legendre elliptic curves 
and the $\frakS_3$-orbits of elements of $\lambda\in\P^1(K)-\{0,1,\infty\}$.

The following result is Lemma 7.1 of \cite{HD}, we state it here for the convenience of the reader. 
\begin{lemma}\label{unitsareintegral}
Suppose the size of $S$ is $2$. Let $(\lambda, \mu)  \in \Lambda_S$, where $\Lambda_S$ is the set of solutions to the 
$S$-unit equation \ref{uniteqn}. Then there is some element $\sigma \in \frakS_3$ such that
$(\lambda, \mu)^\sigma = (\lambda', \mu')$ satisfies $\lambda', \mu' \in \calO_K.$

\end{lemma}

\section{Proof of the Main Theorem}

In this section, we will prove our main theorem.  Our proof closely follows the proof of Theorem 3 in \cite{FS}.
Let $K$ be a number field
satisfying Conjectures \ref{conj1} and \ref{conj2} and $S,U,V$ be the sets in (\ref{sets}) with $U\neq \emptyset$.  Let $B_K$ be as
in Theorem \ref{thm:levelred}, and let $(a,b,c)$ be a non-trivial solution to the Fermat equation \ref{maineqn}, which is
scaled so that $\calG_{a,b,c}=\frakm$ where $\frakm\in \calH$, with exponent $B_K$.  We now apply Theorem \ref{thm:levelred}
and obtain an elliptic curve $E'/K$ having full 2-torsion and potentially good reduction away from $S$ with $j$-invariant
$j'$ satisfying $v_{\frakP}(j')<0$ for all $\frakP\in U$.

Let $\lambda$ be any of the $\lambda$-invariants of $E'$.  As $E'$ has potentially good reduction away from $S$, the
$j$-invariant $j'$
belongs to $\calO_S$ where $\calO_S$ is the ring of $S$-integers in $K$.  From the equation \ref{j-inv}, we can deduce
that $\lambda\in K$ satisfies a monic polynomial with coefficients in $\calO_S$ implying $\lambda\in \calO_S$.  On the
other hand, notice that $1/\lambda, \mu:=1-\lambda$ and $1/\mu$ are solutions of (\ref{j-inv}) and hence elements of
$\calO_S$.  Therefore, we conclude that $(\lambda,\mu)$ is a solution of the $S$-unit equation \ref{uniteqn}.

Now, by the assumption of our main theorem, there is some $\frakP\in U$ for every such solution $(\lambda,\mu)$ 
satisfying $\max\{|v_{\frakP}(\lambda)|,|v_{\frakP}(\mu)|\}\leq 4 v_{\frakP}(2)$.  
Let $t:=\max\{|v_{\frakP}(\lambda)|,|v_{\frakP}(\mu)|\}$ and further let us express $j'$ in terms of $\lambda$
and $\mu$ as:
\begin{equation}\label{newj-inv}
 j'=2^8\frac{(1-\lambda\mu)^3}{(\lambda\mu)^2}.
\end{equation}

Now, if $t=0$, then $v_{\frakP}(j')\geq 8v_{\frakP}(2)\geq0$ by (\ref{newj-inv}).  This contradicts with the 
assumption that $v_{\frakP}(j')<0$ for all $\frakP\in U$.  Hence, we may suppose $t>0$.  The relation
$\lambda+\mu=1$ leads to $v_{\frakP}(\lambda+\mu)\geq \min\{v_{\frakP}(\lambda),v_{\frakP}(\mu)\}$ with equality
if $v_{\frakP}(\lambda)\neq v_{\frakP}(\mu)$.  This shows either $v_{\frakP}(\lambda)=v_{\frakP}(\mu)=-t$, or
$v_{\frakP}(\lambda)=t$ and $v_{\frakP}(\mu)=0$, or $v_{\frakP}(\lambda)=0$ and $v_{\frakP}(\mu)=t$.  Therefore,
$v_{\frakP}(\lambda\mu)=-2t<0$ or $v_{\frakP}(\lambda\mu)=t>0$.  In any of the cases, we have
$v_{\frakP}(j')\geq 8v_{\frakP}(2)-2t\geq 0$, which again yields a contradiction.  This completes the proof of the 
main theorem.

\section{Proof of Density Theorem 1}

To prove the first density result we need to understand the solutions to the $S$-unit equation in the quadratic fields 
$K$ satisfying the assumptions of the theorem. Using the notation introduced in the beginning of the paper we 
have $S=\{\frakP, q \}$ and $T=\{\frakP\}$ where $\frakP^2=2\calO_K$. Note that $2$ is ramified in $K$ by assumption. 
First assume that $d > 2$.  Then the unit group of $\calO_K$ is $\{\pm 1\}$. Also the ideal $\frakP$ cannot be principal 
since the equation $a^2+db^2=2$ has no solutions in integers. 

We claim that if $(\lambda, \mu)$ is a solution to the $S$-unit equation $\lambda+ \mu =1$ then $(\lambda, \mu) $ is in 
the set $\{(-1,2), (1,-2), (1/2, 1/2)\}.$ Following \cite{FS}, these solutions will be called \emph{irrelevant solutions},
while other solutions will be called \emph{relavant solutions}. By Lemma \ref{unitsareintegral} there is an element 
$(\lambda', \mu') $ in the orbit of $(\lambda, \mu) $ such that $\lambda', \mu' \in \calO_K. $

In this section we use Section 7 of \cite{HD}, most of the results are analogous to results therein. For instance we only
give the statement of the following result since its proof is given in \cite{HD}. 

\begin{lemma}\label{rationalunits}
Let $K$ and $S$ be as in the statement of the Density Theorem 1 and let $(\lambda, \mu) \in \Lambda_S$. Then
$\lambda, \mu \in \mathbb Q$ if and only if $(\lambda, \mu )$ belongs to the $\frakS_3$-orbit 
$\{(-1,2), (1,-2), (1/2, 1/2)\}.$
\end{lemma}

This result is following the proof given in \cite{FS} and \cite{HD} for the imaginary case.

\begin{lemma}\label{parametrization}
Up to the action of $\frakS_3$, every relevant $(\lambda, \mu) \in \frakS_3$ has the form: 
\begin{eqnarray}
\lambda=\frac{2^{2r_1}q^{2s_1}-q^{2s_2}+1+v\sqrt{-d}}{2},  \label{eqn:lmu}
\mu=\frac{q^{2s_2}- 2^{2r_1}q^{2s_1}+1-v\sqrt{-d}}{2},  \\ 
\text{where\;} r_1 , s_1, s_2 \geq 0,  s_1s_2=0, v \in \mathbb Z-\{0 \} \\
\text{and satisfies\; }   (2^{2r_1}q^{2s_1}- q^{2s_2} +1)^2 + dv^2= 2^{2r_1+2}q^{2s_1}  \label{eqn:param} \\
(q^{2s_2}- 2^{2r_1}q^{2s_1} +1)^2 + dv^2= 2^{2}q^{2s_2} \label{eqn:param2}
\end{eqnarray}

\end{lemma}

\begin{proof}
It is clear that if $\lambda, \mu$ satisfy the given equations with the given conditions $(\lambda, \mu)$ is a relevant 
element of $\Lambda_S.$

Conversely assume that $(\lambda, \mu)$ is a relevant element of $\Lambda_S.$ By Lemmas \ref{unitsareintegral} and 
\ref{rationalunits} we can assume that $\lambda, \mu \in \mathcal O_K- \mathbb Q.$ Since $S=\{2,q\}$ we can write 
$\lambda = 2^{r_1} q^{s_1} \lambda'$ and $\mu = 2^{r_2} q^{s_2} \mu'$ where $\lambda', \mu' $ are units. 
Since $\lambda+ \mu=1$ we see that $r_1r_2=s_1s_2=0.$ Without loss of generality we can assume that $r_2=0$. Then we get:

$\lambda\bar{\lambda} = 2^{2r_1} q^{2s_1}, \mu\bar{\mu} = q^{2s_2}$ where $\bar{x}$ denotes the conjugation for any 
$x \in K$ and $\lambda+\bar{\lambda}=\lambda\bar{\lambda} -(1-\lambda)(1-\bar{\lambda})+1=\lambda\bar{\lambda}-
\mu\bar{\mu}+1=2^{2r_1}q^{2s_1}-q^{2s_2}+1$. Note that $\lambda- \bar{\lambda} = \nu \sqrt{-d},$ where $\nu \in \mathbb{Z}-
\{0\}$ since $\lambda \notin \mathbb Q$. Combining this with the expression for $\lambda +\bar{\lambda}$ we get the 
equation given in (\ref{eqn:lmu}). Similarly using the equation $\mu=1-\lambda$ we get the similar expression for $\mu$. 
The relations \ref{eqn:param} and \ref{eqn:param2} follow from the identity: $(\lambda + \bar{\lambda})^2 -
(\lambda-\bar{\lambda})^2 = 4 \lambda \bar{\lambda}$ and similar identity for $\mu$.
\end{proof}

The following lemma shows that under the given assumptions there are no relevant solutions of the unit equation and finishes 
 the proof of the first density theorem.

\begin{lemma} \label{norel}
Let $d \geq 7$ be a squarefree integer such that $-d \equiv 2,3 \pmod 8$, $q \geq 29$ be a prime such that $ q \equiv 5 
\pmod 8$ and $\left (\frac{d}{q} \right )=-1.$ Then there are no relevant elements in $\Lambda_S.$
\end{lemma}

\begin{proof}

By Lemma \ref{parametrization} we have $s_1s_2=0$. 
\begin{itemize}
\item Case 1: Say $s_1 \geq 0.$ Then $s_2=0$. Since $q$ doesn't divide $d$, we have by (\ref{eqn:param}) $q^{2s_1}$ divides
$v^2$ i.e. $q^{s_1}$ divides $v$. By dividing both sides of  (\ref{eqn:param}) by $q^{2s_1}$ we get:
$2^{4r_1} q^{2s_1} + d \left ( \frac{v}{q^{s_1}} \right )^2 = 2^{2r_1+2}$. Bringing both sides to modulo $q$ we get:
$d \left ( \frac{v}{q^{s_1}} \right )^2 \equiv 2^{2r_1+2} \pmod q$ i.e. $\left ( \frac{d}{q} \right )=1$ which is a 
contradiction. 

Hence $s_2 \neq 0$ and $s_1=0$.

\item Case 2: We know by Case 1 that $s_1=0 $. Also assume that $s_2=0$.  If $r_1 \geq 1$ then $(2^{2r_1})^2 \geq 2^{2r_1+2}$ 
hence we get a contractiction to equation \ref{eqn:param}. If $r_1=0$ we get $dv^2=3$ by (\ref{eqn:param}), contradiction again.

\item Case 3: We know by Case 1 that $s_1=0$ and by Case 2 that $s_2> 0$. Reducing (\ref{eqn:param2}) modulo $q^{2s_2}$ we get
$(2^{2r_1}+1)^2 \equiv -dv^2 \pmod {q^{2s_2}}$. Since $\left ( \frac{-d}{q} \right )=-1$ we get $q^{s_2} | v$ and
$q^{s_2} | 2^{2r_1}+1.$ Since we assumed that $q \geq 29$ and $q \equiv 5 \pmod 8$, we get $r_1 \geq 5$.  Write 
$v=2^tw, 2 \nmid w$. 
\begin{itemize}
\item Case 3a: Say $r_1-1 \geq t$. Then $2^{2r_1}-q^{2s_2}+1=2^tw'$ where $2 \nmid w'$. Cancelling out $2^{2t}$ from both sides
of (\ref{eqn:param}), we get $w'^2 +dw^2 \equiv 0  \pmod 8$. Since both $w, w'$ are odd, their squares have to be $1$ modulo $8$.
Then we get $-d \equiv 1  \pmod 8,$ contradiction. 
\item Case 3b: By the previous case we can suppose that $t \geq r_1$.  By (\ref{eqn:param}), $2^{r_1}$ divides
$(2^{r_1}-q^{2s_2}+1)$ which implies that $2^{r_1} | (1-q^{s_2})(1+q^{s_2})$. Since $q \equiv 5 \pmod 8$,
$q^{s_2} \equiv 1$ or $5  \pmod 8.$ Then  $2 | (1+q^{s_2}) $ and $4 \nmid (1+q^{s_2})$. Therefore $2^{r_1-1} | (1-q^{s_2})$  
and since $r_1-1 \geq 4$ we get $s_2$ is even, say $2k$. Then similar as above $2^{r_1-1} | q^{2k_2-1}$ implies that
$2^{r_1-2} | (q^{k}-1)$ i.e. $q^k=M 2^{r_1-2} +1$ for some positive integer $M$. We also have seen that
$q^{s_2} | (2^{2r_1} +1)$ which implies that  $q^{s_2}=q^{2k}=M^22^{2r_1-4} +M2^{r_1-1}+1 \leq 2^{2r_1} +1 $ simplifying this 
inequality we get:
$$ M^22^{r_1-3} +M \leq 2^{r_1+1}$$ this implies that $1 \leq M \leq 3.$

Since $q^k=M 2^{r_1-2} +1$ and $r_1 \geq 5$, $q^k \equiv 1 \pmod 8$. Therefore $k$ must be even, say $2u$ and we get
$(q^u-1)(q^u+1)=M 2^{r_1-2}$. The same argument can be applied, since $q \equiv 5 \pmod 8, $ we get $2 | (1+q^{u}) $ and
$4 \nmid (1+q^{u})$.  But then $1+q^{u} = 2 $ or $6$ gives a contradiction and completes the proof. 
\end{itemize}
\end{itemize}
\end{proof}

\section{Proof of Density Theorem 2}

In order to prove the second density result, we would like to understand the solutions of the $S$-unit equation in the 
imaginary quadratic fields $K=\Q(\sqrt{-d})$ for squarefree $d$ such that $-d\equiv 1 \pmod 8 $.  In this case,
note that 2 splits in $K$ and $S=T=U=\{\frakP_1,\frakP_2\}$ where $2\calO_K=\frakP_1\frakP_2$.  Suppose that $d>2$ and
therefore the unit group of $\calO_K$ is $\{\pm1\}$.

\begin{lemma}\label{relevantelts}
 Up to the action of $\frakS_3$, every relevant element $(\lambda,\mu)\in\Lambda_S$ 
 has the form
 
 \begin{eqnarray}
\lambda=\frac{2^{r_1}-2^{r_2}+1+v\sqrt{-d}}{2},  \label{lmu}
\mu=\frac{2^{r_2}- 2^{r_1}+1-v\sqrt{-d}}{2},  \\ 
\text{where\;} r_1\geq r_2 \geq 0, v \in \mathbb Z-\{0 \} \\
\text{and satisfies\; }   (2^{r_1}- 2^{r_2} +1)^2 - 2^{r_1+2}=-dv^2  \label{param} \\
 (2^{r_2}- 2^{r_1} +1)^2 - 2^{r_2+2}=-dv^2\label{param2}
\end{eqnarray}
 In fact, we obtain $r_1=r_2$ if  $r_1,r_2\geq 3$, and equations \ref{param}
 and \ref{param2} become $1-2^{r+2}=-dv^2$ where we set $r_1=r_2=r$.   Also, we have the following exceptional cases:
 $r_1=1, r_2=0$, $r_1=2, r_2=1$, $r_1=3$ and $r_2=2$. Moreover, if $-d\equiv 1 \pmod 8$ then $v$ is an odd 
  integer. 
\end{lemma}

\begin{proof}
 Assume that $\lambda$ and $\mu$ are given by equation \ref{lmu} and satisfy the conditions of the lemma.  Then,
 we see that $\lambda$ and $\mu$ are in $\calO_S$ but not in $\Q^{*}$ and satisfy the $S$-unit equation $\lambda+\mu=1$. 
 Furthermore, $\lambda$ and $\mu$ have norms $2^{r_i}$ by (\ref{param}) and (\ref{param2}), and hence are in $\calO_{S}^{*}$.
 Therefore, $(\lambda,\mu)$ is a relevant element of $\Lambda_S$.
 
 Coversely, suppose that  $(\lambda,\mu)$ is a relevant element of $\Lambda_S$.  By Lemmas \ref{unitsareintegral} and
 \ref{rationalunits}, we can assume that $\lambda,\mu\in\calO_K-\Q$.  Let $\bar{x}$ denotes the conjugation for
 any $x\in K$.  
 
 Now, $\lambda\bar{\lambda}=2^{r_1}$ and $\mu\bar{\mu}=2^{r_2}$ and we may assume $r_1\geq r_2\geq 0$.  We note that
 \[
  \lambda+\bar{\lambda}=\lambda\bar{\lambda}-(1-\lambda)(1-\bar{\lambda})+1=\lambda\bar{\lambda}-\mu\bar{\mu}+1
  =2^{r_1}-2^{r_2}+1
 \]
and $\lambda-\bar{\lambda}=v\sqrt{-d}$ where $v\in \mathbb{Z}-\{0 \}$ since $\lambda\notin\Q$.  By combining the 
previous expressions, we get the equation given in (\ref{lmu}).  The equations \ref{param} and \ref{param2} follow
from the identity $(\lambda+\bar{\lambda})^2-(\lambda-\bar{\lambda})^2=4\lambda\bar{\lambda}$ and the corresponding
identity for $\mu$.

Now, assume $r_1> r_2\geq 3$ and $r_1=r_2+a$ where $a\geq 1 $.  
Then from the equations \ref{param} and \ref{param2}, we have 
\[
(2^{r_2+a}-2^{r_2}+1)^2<2^{r_2+a+2}\mbox{ and } (2^{r_2}-2^{r_2+a}+1)^2<2^{r_2+2}.
\]
By expanding the squares and adding the two inequalities we get:


\[
 2^{2r_2+2a+1}+2^{2r_2+1}-2^{2r_2+a+2}<2^{r_2+a+2}+2^{r_2+2}-2<2^{r_2+a+2}+2^{r_2+2}. 
\]

In other words, we have 
\[
 2^{2r_2+2a+1}+2^{2r_2+1}-2^{2r_2+a+2}<2^{r_2+2}(2^a+1).
\]
Now, after dividing both sides by $2^{r_2+2}$ and rearranging we get:
\[
 2^{r_2-1}(2^a-1)^2<2^a+1.
\]

Since $r_2\geq 3$, then $2^{\frac{r_2-1}{2}}\geq 2^{1}=2$ and $\sqrt{2^a+1}<2^a$ when $a\geq 1$.  Hence, we obtain
\[
  2(2^a-1)<2^a \rightarrow 2^a<2,
\]
which is possible only when $a=0$.  This gives a contradiction and therefore we get $a=0$, i.e. $r_1=r_2\geq 3$.
Exceptional cases can be verified easily.  Finally, if we reduce the equation $1-2^{r+2}=-dv^2$ modulo $8$, we get
$1\equiv v^2 \pmod 8$ implying $v$ is a nonzero odd integer.  Note that if $-d\equiv 5 \pmod 8$ or  
$-d\equiv 2 \mbox{ or } 3 \pmod 4$ there are no relevant solutions.

 \end{proof}

Now we will show that under the assumptions of Density Theorem 2, the solutions of the unit equation have the necessary 
assumptions to apply the Main Theorem i.e. a solution $(\lambda, \mu)$ satisfies
$\max\{|v_\frakP(\lambda)|, |v_\frakP(\mu)|\} \leq 4v_\frakP(2)$ where $ \frakP $  is a prime lying over $2$ with inertia 
degree $1$. This will complete the proof of Density Theorem 2. 

Since $K=\Q(\sqrt{-d})$ and $-d \equiv 1 \pmod 8$, $2$ splits in $\mathcal O_K$ and hence we need to verify that
$\max\{|v_\frakP(\lambda)|, |v_\frakP(\mu)|\} \leq 4v_\frakP(2)$ for any solution $(\lambda, \mu)$ of the $S$-unit 
equuation $\lambda+\mu =1$. Note that it is enough to check this for any representative of an orbit under the action of 
$\frakS_3$ by Lemma 6.2(i) of \cite{FS}.  Assume that $\max\{|v_\frakP(\lambda)|, |v_\frakP(\mu)|\}= |v_\frakP(\lambda)|=r > 4.$ By Lemma \ref{relevantelts}, 
$\lambda=\frac{1+v \sqrt{-d}}{2}$ where $1+dv^2=2^{r+2}.$ 

 Consider $1+dv^2=2^{r+2}$ modulo $6$. Since $2^{r+2} \equiv 2$ or $4 \pmod 6$ we get $dv^2 \equiv 1$ or $3 \pmod 6$. 
Using the assumption that $d \equiv 5 \pmod 6$ and  multiplying both sides with $5$ we get $v^2 \equiv 5$ or $3 \pmod 6.$ 
Since  
$v^2 \equiv 5 \pmod 6$ is not possible we get $v^2 \equiv 3 \pmod 6$ i.e. $v$ is divisible by $3$ and $v^3$ is divisible by $9$. Now using this information and reducing the equation $1+dv^2=2^{r+2}$ modulo $9$ we get $1 \equiv 2^{r+2} \pmod 9$ 
which implies that $6 | (r+2).$ 

To finish the proof we consider $1+dv^2=2^{r+2}=2^{6k}$ modulo $14$. Since $k >1$ we get $1+dv^2=2^{r+2} \equiv 8 \pmod {14}$ i.e. 
$dv^2 \equiv 1 \pmod {14} $. Using the fact that $v^2 \equiv 0,1,2,4,8,9,11 \pmod {14} $ we get that the only solution is when
$d \equiv 7 \pmod {14}$, which is a contradiction to the second assumption of Density Theorem 2.

\section{Proof of Density Theorem 3}
 Our reference for the notation used in this section is \cite{FS}.  We will rewrite it here for the convenience of
 the reader.
 
 Let $\calU$ be a set of positive integers and $X$ be a positive real number.  Then, define 
 $\calU(X)=\{d\in \calU : d\leq X\}$.  If the limit
 $ \delta(\calU)=\lim\limits_{X\to\infty}\frac{\#\calU(X)}{X}$
exists, it is called the \emph{absolute density} of $\calU$.

Define $\N^{sf}=\{d \geq 2: d \mbox{ squarefree}\}$.  If the limit
\[
 \delta_{rel}(\calU)=\lim_{X\to\infty}\frac{\#\{d\in \calU: d\leq X\}}{\#\{d\in\N^{sf}:d\leq X\}}
 =\lim_{X\to\infty}\frac{\#\calU(X)}{\#\N^{sf}(X)}
\]
for a subset $\calU\subset \N^{sf}$ exists, then it called the \emph{relative density of $\calU$ in $\N^{sf}$}.

Throughout this section, we will use the following two classical theorems from analytic number theory.

\begin{theorem}(e.g [Lan09,p.636].\label{Landau1})
 For integers $r$ and $N$ with $N$ positive, let $ \N_{r,N}^{sf}=\{d\in \N^{sf}: d\equiv r \pmod N\}.$
Let $s=\gcd(r,N)$ and suppose that $s$ is squarefree.  Then
\[
 \#\N_{r,N}^{sf}(X)\sim \frac{\varphi(N)}{s\varphi(N/s)N\prod_{q|N}(1-q^{-2})}\frac{6}{\pi^2}X,
\]
where $\varphi$ denotes the Euler's totient function.
\end{theorem}

\begin{theorem}(e.g. [Lan09,pp.641-643]).\label{Landau2}
Let $N$ be a positive integer.  Let $r_1,\dots,r_m$ be distinct modulo $N$, satisfying $\gcd(r_i,N)=1$.  Let $E$ be
the set of positive integers $d$ such that every prime factor $q$ of $d$ satisfies $q\equiv r_i \pmod N$ for some
$i=1,\dots k$.  Then there is some positive constant $\gamma=\gamma(N,r_1,\dots,r_m)$ such that
\[
 \#E(X)\sim\gamma\frac{X}{\log(X)^{1-m/\varphi(N)}}.
\]
\end{theorem}

\begin{remark} \label{densityremark}Observe that if we apply Theorem \ref{Landau1} to $\N^{sf}=\N_{0,1}^{sf}$, we obtain
$\#\N^{sf}(X)\sim 6X/\pi^2$.
Moreover, for a subset $\calU$ of $\N^{sf}$, $\delta(\calU)$ exists if and only if $\delta_{rel}(\calU)$ exists and
$\delta_{rel}(\calU)=\pi^2\delta(\calU)/6$.
\end{remark}
We would like to prove the following theorem.
\begin{theorem}\label{relativedensity}
 Let $\calC$ be the set of $d \in \N^{sf}$ such that  the $S$-unit equation \ref{uniteqn} has no relevant solutions in $\Q(\sqrt{-d})$ 
 and $ \calD=\{d\in \calC : -d \not \equiv 5 \pmod 8\}.$ 
 Then $\delta_{rel}(\calC)=1$ and $\delta_{rel}(\calD)=5/6$.

\end{theorem}

In order to prove Theorem \ref{relativedensity}, we need the following lemma.

\begin{lemma}\label{complementlemma}
 Let $\calC'=\N^{sf}\backslash \calC$.  Then $\delta(\calC')=0$.
\end{lemma}

By assuming Lemma \ref{complementlemma}, we can prove Theorem \ref{relativedensity} as below.

\begin{proof4} By Lemma \ref{complementlemma} and Remark \ref{densityremark}, we have 
$\delta_{rel}(\calC')=\delta(\calC')=0$.
The fact that $\calC$ and $\calC'$ are complements in $\N^{sf}$ yields $\delta_{rel}(\calC)=1$. 

Note that $\calD=\calC \cap (\N^{sf}-\N_{3,8}^{sf})$.  By Theorem \ref{Landau2}, $\#\N_{5,8}^{sf}\sim X/\pi^2$
and by Remark \ref{densityremark} $\#\N^{sf}\sim 6X/\pi^2$. Hence, $\delta_{rel}(\calD)=5/6$.
\end{proof4}

 \subsection{Proof of Lemma \ref{complementlemma}}
 In this section, we give the proof of Lemma \ref{complementlemma}.  Throughout this section, we follow the proof of Lemma 7.1
 of \cite{FS}.  As we deal with the imaginary quadratic fields, we use a simplified version of their argument.
 
 Being the complement of $\calC$, the set $\calC'$ consists of squarefree $d\geq2$ such that the $S$-unit equation 
\ref{uniteqn} has a relevant solution in $\Q(\sqrt{-d})$.  By Lemma \ref{relevantelts}, this set exactly contains
squarefree $d\geq 2$ satisfying $1-2^{r+2}=-dv^2$ where $r\geq 3$ and $v$ is an odd integer.  Now, for $\kappa_{1}=0$ and
$\kappa_{2}=1$, we define $\calC'(\kappa_1)$ and $\calC'(\kappa_2)$ as the sets containing squarefree $d\geq 2$ satisfying
$1-2^{r+2}=-dv^2$ where $r\geq 3$ and $v$ is an odd integer with $r\equiv \kappa_1$ and $\kappa_2 \pmod 2$ respectively.  Hence,
$\calC'=\calC'(\kappa_1)\cup\calC'(\kappa_2)$.  Assume $q$ is an odd divisor of $d$ and let $d\in \calC'(\kappa_2)$.  
Then reducing $1-2^{r+2}$ modulo $q$ implies that $2$ is a quadratic residue modulo $q$.  In this case, $q$
is an element of a proper subset of the congruence classes $\{\bar{1},\bar{3},\bar{5},\bar{7}\}$ modulo $8$.  Now by using
Theorem \ref{Landau2} (note that in this case, the power of the logarithm in the theorem becomes less than 1), we
see that $\delta(\calC'(\kappa_2))=0$

Now we will show that $\delta(\calC'(\kappa_1))=0$. Assume $d\in \calC'(\kappa_1)$ and set $r=2s, s\geq 2$.  The condition for
the relevant solutions becomes: $(2^{s+1}+1)(2^{s+1}-1)=dv^2. $

Let  $ \alpha_{1,s}=2^{s+1}+1\mbox{ and } \alpha_{2,s}=2^{s+1}-1.$ Obviously, $\alpha_{1,s}$ and $\alpha_{2,s}$ are odd and coprime integers.  Therefore, we can write
$\alpha_{i,s}=d_{i,s}v_{i,s}^2$ where $d_{i,s}$ are squarefree and $d=\prod d_{i,s}$.  Hence,
\[
 \calC'(\kappa_1)=\{d_{1,s}d_{2,s}:s\geq 2\}.
\]
 Instead showing $\delta(\calC'(\kappa_1))=0$, we will prove the equivalent statement that 
 \[
  \delta_{sup}(\calC'(\kappa_1))=\limsup_{X\to\infty}\#\calC'(\kappa_1)(X)/X=0.
 \]
 The $m^{\mbox{th}}$ Mersenne number $M_m$ is defined to be $M_m=2^m-1$ for any positive integer $m$.  Note that
 if $n|m$, then $M_n|M_m$.
 
 \begin{lemma}\label{Mersenne}
  Let $m$ be a positive integer.  Let $M_m=2^m-1$.  Assume that $s_1\equiv s_2 \pmod m$.  Then
  $\alpha_{1,s}\equiv \alpha_{2,s} \pmod {M_m}$ for $i=1,2$.
 \end{lemma} 
  \begin{proof} Assume without loss of generality that  $s_1>s_2$.  Since $s_1\equiv s_2 \pmod m$, $m|(s_1-s_2)$. Then,
  $ \alpha_{i,s_1}-\alpha_{i,s_2}=2^{s_1+1}-2^{s_2+1}=2^{s_2+1}(2^{s_1-s_2}-1). $
  Since $m|(s_1-s_2)$ we get $M_m|(\alpha_{i,s_1}-\alpha_{i,s_2})$.
   
  \end{proof}

 \begin{lemma}\label{approx1}
   Let $m$ be a positive integer and $s_0\geq 2$ and  $\calA_{m,s_0}=\{d_{1,s} : s\equiv s_0\pmod m \}. $
 Let $p_1,\dots,p_k$ be distict primes dividing $M_m$ that do not divide $\alpha_{1,s_0},$ and write
 $N=p_1 \dots p_k$.  Then
 \[
  \# \calA_{m,s_0}(X)\leq 2^{-k}X+N.
 \]
 \end{lemma}
\begin{proof}
 Assume $s\equiv s_0 \pmod m$.  Since $N|M_n$, $\alpha_{1,s}\equiv \alpha_{1,s_0}$ by Lemma \ref{Mersenne}.  Moreover,
 $\alpha_{1,s}$ is coprime to $N$ as $\alpha_{1,s_0}$ is coprime to $N$.  Hence, $v_{1,s}^2$ is coprime to $N$ where
 $\alpha_{i,s}=d_{i,s}v_{i,s}^2$.  By reducing modulo $N$ we obtain $d_{1,s}\equiv \alpha_{1,s_0}v_{1,s}^{-2} \pmod N$,
 i.e. $d_{1,s}\in\{\alpha_{1,s_0}\omega^2|\omega\in(\Z/N\Z)^*\}$.  This set has cardinality $\varphi(N)/2^k$ since $N$
 is squarefree with $k$ distinct prime factors.  Then the lemma follows.
\end{proof}

Let us denote by $\omega(n)$ the number of distinct prime divisors of a positive integer $n$.

\begin{lemma}\label{approx2}
 For $m\geq 1$, let $h_m=\omega(M_m)$.  Then $\delta_{sup}(\calC'(\kappa_1))\leq 2^{-h_m/2}m$.
\end{lemma}

\begin{proof}
 Let $k$ and $N$ as in Lemma \ref{approx1} for some fixed $s_0$.  Let 
 \[
  \calC'(\kappa_1)_{m,s_0}=\{d_{1,s}d_{2,s}:s\geq 2, s\equiv s_0 \pmod m\}.
 \]
Then, clearly 
\[
 \# \calC'(\kappa_1)_{m,s_0}(X)\leq \#\calA_{m,s_0)}(X)\leq 2^{-k}X+N\leq 2^{-k}X+M_m.
\]
 Let $q_1,\dots,q_k'$ be the distinct primes dividing $M_m$ that do not divide $\alpha_{2,s_0}$.  A similar
 argument to Lemma \ref{approx1} gives
 \[
  \# \calC'(\kappa_1)_{m,s_0}(X)\leq 2^{-k'}X+M_m.
 \]

 Note that $\alpha_{1,s_0}$ and $\alpha_{2,s_0}$ are coprime implying that either $k\geq h_m/2$ or $k'\geq h_m/2$.
 Therefore,
 \[
  \# \calC'(\kappa_1)_{m,s_0}(X)\leq 2^{-h_m/2}X+M_m.
 \]
If we set $s_1,\dots,s_m$ be a complete system of representatives for $s$ modulo $m$, then
\[
 \# \calC'(\kappa_1)(X)\leq 2^{-h_m/2}Xm+mM_m.
\]

\end{proof}




\begin{lemma}
 With notation as above, $\delta(\calC'(\kappa_1))=0$.
\end{lemma}

\begin{proof}
 By Lemma \ref{approx2}, we have
 \begin{equation}\label{deltasup}
  0\leq\delta_{sup}(\calC'(\kappa_1))\leq\frac{m}{2^{(h_m/2)}}\leq \frac{m}{2^{(2^{\omega(m)-1})-1}}
 \end{equation}
for any $m\geq 1$, where the last inequality follows from the fact that $h_m\geq 2^{\omega{m}}-2$ [Corollary 7.5 in \cite{FS}].
Assuming $y$ is large and setting $m=\prod_{p\leq y}p$, note that $\omega(m)=\pi(y)$ and 
$m=\exp(\vartheta(y))$, where $\pi$ and $\vartheta$ are the prime counting function and the first Chebyshev function.  Since
\[
 \pi(y)\sim y/ \log y
\mbox{ and } \vartheta(y)\sim y
\]
by the prime number theorem (see e.g., [Apo76, ch. 4]), we obtain $\delta_{sup}(\calC'(\kappa_1))=0$ as $y\to \infty$ in 
\ref{deltasup}.
\end{proof}

\Addresses
\end{document}